\documentclass{article}

\usepackage{multirow}
\usepackage{subcaption}
\usepackage{mathtools}
\newcommand{\vect}[1]{\underline{\boldsymbol{#1}}}
\usepackage[utf8]{inputenc}
\usepackage{hyperref}
\usepackage{longtable}
\usepackage{graphics,setspace}
\usepackage{array,multirow}
\usepackage{tikz,float}
\usepackage{amsmath,amsthm,amssymb}
\usepackage{enumerate,enumitem,listings}

\usepackage[T1]{fontenc}

\newcommand{\hookdoubleheadrightarrow}{
  \hookrightarrow\mathrel{\mspace{-15mu}}\rightarrow
}
\newtheorem{theoremalph}{Theorem}

\newtheorem{coralph}[theoremalph]{Corollary}

\newtheorem{theorem}{Theorem}[section]
\newtheorem{lemma}[theorem]{Lemma}

\newtheorem*{theorem*}{Theorem}
\newtheorem*{lemma*}{Lemma}
\newtheorem*{proposition*} {Proposition}

\theoremstyle{definition}
\newtheorem{definition}[theorem]{Definition}
\newtheorem{remark}[theorem]{Remark}

\setlist[enumerate]{align=left}
\usepackage{enumitem}
\usepackage{caption} 
\usepackage{relsize}
\allowdisplaybreaks
 \usepackage{ragged2e} 
 \usepackage{longtable}

\title{Link conditions for the Haagerup property}
\author{Calum J. Ashcroft}
\date{\vspace{-2em}}

\begin{document}
\maketitle
	\begin{abstract}
We provide a condition on the links of a polygonal complex $X$ that is sufficient to ensure $Aut(X)$ has the Haagerup property, and hence so do any closed subgroups of $Aut(X)$ (in particular, any group acting properly on $X$). We provide an application of this work by considering the group of automorphisms of simply-connected triangle complexes where the link of every vertex is isomorphic to the graph $F090A$, as constructed by \'{S}wi\k{a}tkowski. \end{abstract}

\section{Introduction}
Consider a topological group $G$ acting on a polygonal complex $X$, i.e. a two-dimensional metric polyhedral complex. One can ask how local properties of $X$ can give large scale information about the group-theoretic and geometric properties of $G$. It is often most natural to consider properties relating to the \emph{link} of vertices in $X$. For a polygonal complex $X$ and a vertex $v$ we define the \emph{link} of $v$, $Lk_{X}(v)$ (or simply $Lk(v)$ when $X$ is clear from context), as the graph whose vertices are the edges of $X$ incident at $v$, and two vertices $e_{1}$ and $e_{2}$ are connected by an edge $f$ in $Lk(v)$ if the edges $e_{1}$ and $e_{2}$ in $X$ are adjacent to a common face $f$. We can endow the link graph with the \emph{angular metric}; an edge $f=(e_{1},e_{2})$ in $Lk(v)$ has length $\alpha$, where $\alpha$ is the angle between $e_{1}$ and $e_{2}$ at $v$ in the shared face $f$. 

\emph{\.{Z}uk's criterion} gives a local condition to prove a group has Property (T): if $X$ is a simply-connected triangular complex, $G$ acts on $X$ properly discontinuously and cocompactly, and $\lambda_{1}(Lk_{X}(v))>1\slash 2$ for every vertex $v$, then $G$ has Property (T) \cite{zuk1996,Ballmann-Swiatkowski}.

The purpose of this work is to provide a sufficient local condition, which should be often (computationally) checkable in practice, to prove that a group has the Haagerup property, a strong negation of Property (T). In particular, we study the property of polygonal complexes of being \emph{evenly $\pi$-separated}.

We will assume that for any polygonal complex $X$, the set of isometry classes of faces in $X$, $Shapes(X)$, is finite, and also that $X$ is \emph{proper}, i.e. closed balls are compact. Importantly, since $X$ is a proper metric space, it is locally-compact. Of course, if $X$ is CAT(0), then it is simply-connected. Furthermore, since $X$ is second-countable, Hausdorff, and locally-compact, so is $Aut(X)$. Our main theorem is as follows.

\begin{theoremalph}\label{mainthm: link conditions for the Haagerup property}
Let $X$ be an evenly $\pi$-separated CAT(0) polygonal complex. Then $Aut(X)$ has the Haagerup property.
\end{theoremalph}

Since the Haagerup property passes to closed subgroups, we can immediately deduce the following.
\begin{coralph}\label{maincor: link conditions for the Haagerup property for closed subgroups}
Let $X$ be an evenly $\pi$-separated CAT(0) polygonal complex and let $G$ be a group acting properly on $X$. Then $G$ has the Haagerup property.
\end{coralph}

In the opposite direction to Property (T); if $G$ acts on a CAT(0) polygonal complex $X$ properly discontinuously and cocompactly, and for each vertex $v$ the edges of $Lk_{X}(v)$ can be partitioned into two-sided \emph{$\pi$-separated cutsets}, then $G$ acts properly discontinuously on a CAT(0) cube complex (this action is also cocompact if $G$ is hyperbolic) \cite[Example 4.3]{Hruska-Wise}. The requirement that the cut sets are partitions was weakened in \cite{Ashcroft2020LinkCF} to the requirement that the cut sets satisfy certain gluing equations and are \emph{equatable}. The existence of a properly discontinuous and cocompact action of a group $G$ on $X$ is not clear in many cases, and so this does not provide an effective way to negate Property (T) in $Aut(X)$. Examples of such complexes are those constructed in \cite{Swiatkowski} with every link isomorphic to the graph $F090A$. There are discrete groups that have the Haagerup property but do not act properly on any CAT(0) cube complex. Some examples of such groups are provided by complex hyperbolic lattices in dimensions greater than $2$ \cite{bekkadelaharpe,delzant2019cubulable}. Importantly, it should be easier to show that a group has the Haagerup property than finding a proper action on a CAT(0) cube complex, since such an action implies the Haagerup property \cite{Cheriz-Martin-Valette_haagerupproperty}.

It is therefore natural to consider weakening the requirements of \cite[Theorem A]{Ashcroft2020LinkCF}. We do this by removing the requirement of a properly discontinuous and cocompact action. However, we still require a solution to the gluing equations, and furthermore, we now require this solution to be $Aut(X)$-invariant.

We provide an application of Theorem \ref{mainthm: link conditions for the Haagerup property} by considering two of the complexes constructed in \cite{Swiatkowski}. The graph F090A, often known as the \emph{Foster graph}, is the distance-regular cubic graph on $90$ vertices appearing in \cite{fostercensus}.

\begin{coralph}\label{maincor: application to F090A complexes}
Let $X$ be a simply-connected triangular complex such that every triangle is a unit equilateral Euclidean triangle and the link of every vertex is isomorphic to $F090A$ as non-metric graphs. Then $Aut(X)$ has the Haagerup property.
\end{coralph}
In fact,  there are (up to isomorphism) exactly two triangular complexes  $X$ as above. For either of the two complexes, Aut(X) is \emph{flexible}, so that vertex stabilizers are uncountable and the compact-open topology is non-discrete \cite{Swiatkowski}.

\subsection{Structure of the paper}

This paper is structured as follows. We define the required conditions on graphs and complexes in Section \ref{sec: Link conditions}. In Section \ref{measured wallspaces}, we remind the reader of the definition of spaces with measured walls, as introduced by \cite{Cheriz-Martin-Valette_haagerupproperty}. In Section \ref{section: Constructing hypergraphs in polygonal complexes}, we build separating convex trees in polygonal complexes, which we then use to construct a space with measured walls. In Section \ref{section: Haagerup property for groups acting on polygonal complexes}, we prove that the resulting action on the wall space is proper and apply a result of \cite{Cheriz-Martin-Valette_haagerupproperty} to prove Theorem \ref{mainthm: link conditions for the Haagerup property}. We then apply Theorem \ref{mainthm: link conditions for the Haagerup property} to some specific polygonal complexes in Section \ref{section: applications} in order to deduce Corollary \ref{maincor: application to F090A complexes}.

\section*{Acknowledgements}
I would like to thank Pierre-Emmanuel Caprace for asking whether the techniques of \cite{Ashcroft2020LinkCF} could be applied to the complexes constructed by \'{S}wi\k{a}tkowski, leading to this work, as well as for comments on an earlier draft of this work. As always, I would like to thank Henry Wilton for his support, guidance, and helpful discussions.

\section{Some conditions on links}\label{sec: Link conditions}
We begin by discussing cutsets in graphs; most of the following arises in \cite{Ashcroft2020LinkCF}. The \emph{combinatorial metric} on a graph $\Gamma$ is the path metric induced by assigning each edge of $\Gamma$ length $1$. We will be using Serre graphs.
\begin{definition}
	Let $\Gamma$ be a finite metric graph.
	\begin{enumerate}[label=\roman*)]
	\item A vertex $v$ (respectively edge $e$) is a \emph{cut vertex} (respectively \emph{cut edge}) if $\Gamma - \{v\}$ (respectively $\Gamma -\{e\}$) is disconnected as a topological space.
		    \item A set $C\subseteq \Gamma$ is a \emph{cutset} if $\Gamma - C$ is disconnected as a topological space. It is an \emph{edge cutset} if $C\subseteq E(\Gamma)$ and a \emph{vertex cutset} if $C\subseteq V(\Gamma).$
	    \item An edge cutset $C$ is \emph{proper} if for any edge $e\in C$, the endpoints of $e$ lie in distinct components of $\Gamma-C$. A vertex cutset $C$ is \emph{proper} if for any vertex $u\in C$, and any distinct vertices $v,w$ adjacent to $u$, the vertices $v$ and $w$ lie in distinct components of $\Gamma-C$.
	    \item  For an edge $e$ in $\Gamma$ let $m(e)$ be the midpoint of $e$, and for $v$ a vertex let $m(v)=v$. For $\sigma >0$, a set $C\subseteq V(\Gamma)\sqcup E(\Gamma)$ is \emph{$\sigma$-separated} if for all distinct $\alpha ,\beta \in C$, $d_{\Gamma}(m(\alpha),m(\beta))\geq \sigma .$
	   \end{enumerate}

	   Given a cutset $C$, we can assign a partition $\mathcal{P}(C)$ to $\pi_{0}(\Gamma-C)$; we \emph{always require} that such a partition is at least as coarse as connectivity in $\Gamma - C$, and each partition contains at least two elements. The \emph{canonical partition} of $C$ is that induced by connectivity in $\Gamma -C$.
\end{definition}

\begin{definition}[\emph{Edge separated}]
	    
Let $\Gamma$ be a finite metric graph, and let $\sigma>0$. We will say that $\Gamma$ is \emph{edge $\sigma$-separated} if $\Gamma$ is connected, contains no vertices of degree $1$, and there exists a collection of proper $\sigma$-separated edge cutsets $C_{i}\subseteq E(\Gamma)$, with $\vert C_{i}\vert\geq 2$ for each $i$ and $\cup_{i}C_{i}=E(\Gamma)$. 
\end{definition}

\begin{definition}[\emph{Vertex separated}]
   Let $\Gamma$ be a finite metric graph and let $\sigma>0$. We will say that $\Gamma$ is \emph{vertex $\sigma$-separated} if:
    \begin{enumerate}[label=\roman*)]
        \item $\Gamma$ is connected and contains no vertices of degree $1$,
        \item there exists a collection of $\sigma$-separated vertex cutsets $C_{i}\subseteq V(\Gamma)$, such that $\vert C_{i}\vert\geq 2$ for each $i$ and $\cup_{i}C_{i}=V(\Gamma)$,
        \item for any vertex $v$ and any distinct vertices $w,w'$ adjacent to $v$ there exists a $\sigma$-separated vertex cutset $C_{i}$ such that $w$ and $w'$ lie in distinct components of $\Gamma - C_{i}$, and 
        \item for any points $u$ and $v$ in $\Gamma$ with $d(u,v)\geq \sigma$, there exists a cutset $C_{i}$ with $u$ and $v$ lying in distinct components of $\Gamma - C_{i}$.
    \end{enumerate}
    	
\end{definition} 
Note that, importantly, in general we do not require vertex cutsets to be proper. This definition is not as difficult to verify as it may seem; importantly, we need only consider separating vertices $u,v$ with $d(u,v)\geq \sigma,$ rather than arbitrary points.
\begin{lemma}\label{lem: vertex separated condition}\cite[Lemma 2.10]{Ashcroft2020LinkCF}
  Let $n\geq 2$, and let $\Gamma$ be a graph endowed with the combinatorial metric, such that $\Gamma$ is connected, contains no vertices of degree $1$, and $girth(\Gamma)\geq 2 n$. Suppose there exists a collection of $n$-separated vertex cutsets $\mathcal{C}=\{C_{1},\hdots , C_{m}\}$ such that: 
  \begin{enumerate}[label=$\roman*)$]
      \item  $\cup_{i}C_{i}=V(\Gamma)$,
      \item  $\vert C_{i}\vert \geq 2$ for each $i$,
      \item for each vertex $v$ and distinct $w$, $w'$ adjacent to $v$ there exists a $n$-separated cutset $C_{i}$ with $w$ and $w'$ lying in distinct components of $\Gamma - C_{i}$, and 
      \item furthermore that for any pair of vertices $u,v$ with $d_{\Gamma}(u,v)\geq n$ there exists a cutset $C_{i}$ with $u$ and $v$ lying in distinct components of $\Gamma -C_{i}$.
  \end{enumerate}
 Then $\Gamma$ is vertex $n$-separated with the collection of cutsets $\mathcal{C}$.
\end{lemma}

Next we extend these definitions to CAT(0) polygonal complexes. We will say that a polygonal complex is \emph{regular} if all polygonal faces are regular polygons, i.e. for any polygon, all of its internal angles are equal. We first define the following graph, which appears in e.g. \cite{Wise-small-cancellation}.

\begin{definition}[\emph{Antipodal graph}]
Let $X$ be a regular non-positively curved polygonal complex. Subdivide edges in $X$ and add vertices at the midpoints of edges; call these additional vertices \emph{secondary vertices}, and call the other vertices \emph{primary}. Every polygon in $X$ now contains an even number of edges in its boundary. Construct a graph $\Delta_{X}$ as follows. Let $V(\Delta_{X})=V(X)$ and join two vertices $v$ and $w$ by an edge, labelled $f$, if $v$ and $w$ exist and are antipodal in the boundary of a face $f$ in $X$; add as many edges as such faces exist. This is the \emph{antipodal graph} for $X$.
\end{definition}

\begin{remark}
We note that for a secondary vertex $s$ of $X$, $Lk_{X}(s)$ is a cage graph with edges of length $\pi$. Hence, if $X$ does not contain any free faces, $Lk_{X}(s)$ is edge $\pi$-separated, with the single $\pi$-separated cutset $E(Lk_{X}(s))$.

As the complex is regular, the edges of $\Delta_{X}$ pass through the midpoints of edges in $Lk_{X}(v)$ for vertices $v$. We can therefore identify (the midpoints of) edge cutsets in $Lk_{X}(v)$ with sets of edges adjacent to $v$ in $\Delta_{X}$, or, equivalently, with subsets of $V(Lk_{\Delta_{X}}(v))$. We will implicitly make this identification for the remainder of this text.

There is a canonical map $\Delta_{X} \rightarrow X$; we map a vertex $v$ of $\Delta_{X}$ to the corresponding vertex of $X$, and we map an edge $e$ labelled by $f$ to the local geodesic between the endpoints of $e$ lying in the face $f$.
\end{remark}

The following is extremely similar to the `splicing' of Manning \cite{Manning10}; we will use this to build our subcomplexes inductively, ensuring they are separating in a manner similar to \cite{Cashen-Macura2011}. Let $\Delta=\Delta_{X}$ if we are dealing with edge cutsets or $\Delta=X^{(1)}$ if we are dealing with vertex cutsets. Suppose that for each vertex $v$ of $\Delta$, we have found a set of cutsets (each identified with a subset of $V(Lk_{\Delta}(v))$) and associated partitions $\{(C,P_{C,i})\}_{C,i}$. Let
$$\mathcal{CP}_{v}=\{(C,P_{C,i})\;:\;C\mbox{ is a }\pi\mbox{-separated cutset},\;C\subseteq Lk(v),\;\vert C\vert\geq 2\}.$$ For an oriented edge $e$ of $\Delta$, we define
$\mathcal{CP}(e):=\{(C,P)\in \mathcal{CP}_{\iota (e)}\;:\;e\in C\}.$

\begin{definition}[Equatable partitions]
Let $X$ be a non-positively curved polygonal complex, and let $\Delta=\Delta_{X}$ if we are dealing with edge cutsets and $\Delta=X^{(1)}$ if we are dealing with vertex cutsets. Choose $\kappa>0$ so that the length of any edge of $\Delta$ is at least $1000\kappa $ (we can do this as $Shapes(X)$ is finite). Let $v,w$ be two vertices of $\Delta$ connected by an oriented edge $e$, so that $v=\iota (e)$ (i.e. the start vertex of $e$) and $w=\tau (e)$ (i.e. the end vertex of $e$). Let $(C_{v},P_{v})\in \mathcal{CP}(e)$ and $(C_{w},P_{w})\in \mathcal{CP}(e^{-1}).$
 
 Let $v'$, $w'$ be points on $e$ in $X$ in an $\kappa$-neighbourhood of $v$, $w$ respectively, so that there are canonical mappings
 \begin{align*}
         i_{v}&:St(v')\hookrightarrow Lk(v),\\
         i_{w}&:St(w')\hookrightarrow Lk(w),\\
         \phi&:St(v')\xrightarrow{\cong}St(w').
     \end{align*}
 Therefore we have induced mappings
 \begin{align*}
         \overline{i}_{v}&:St(v')-v'\hookrightarrow Lk(v)-C_{v},\\
         \overline{i}_{w}&:St(w')-w'\hookrightarrow Lk(w)-C_{w},\\
         \overline{\phi}&:St(v')-v'\xrightarrow{\cong}St(w')-w'.
     \end{align*}
For $u=v,w$, let $\mathcal{P}_{u}$ be the set of all partitions of $\pi_{0}(Lk(u)-C_{u})$, and let $\mathcal{P}_{u'}$ be the set of all partitions of $\pi_{0}(St(u')-u')$. There are induced maps 
  \begin{align*}
         \iota_{v}&:\mathcal{P}_{v}\rightarrow \mathcal{P}_{v'},\\
         \iota_{w}&:\mathcal{P}_{w}\rightarrow \mathcal{P}_{w'},\\
         \psi&:\mathcal{P}_{v'}\hookdoubleheadrightarrow \mathcal{P}_{w'}.
     \end{align*}
The map $\psi$ exists because the map $\overline{\phi}$ above is an isomorphism. We say that $(C_{v},P_{v})$ and $(C_{w},P_{w})$ are \emph{equatable along $e$}, written $$(C_{v},P_{v})\sim_{e}(C_{w},P_{w}),$$
 if $\psi(\iota_{v}(P_{v}))=\iota_{w}(P_{w}).$ Note that this also defines an equivalence relation on $\mathcal{CP}(e)$; for $(C,P),(C',P')\in \mathcal{CP}(e)$, we write $$(C,P)\approx_{e} (C',P')$$ if $\iota_{v}(P)=\iota_{v}(P').$
We let $[[C,P]]_{e}$ be the equivalence class of $(C,P)$ under the equvalence relation $\approx$. We define $[C,P]_{e^{-1}}$ to be the equivalence class of cutset partitions in $\mathcal{CP}(e^{-1})$ equatable to $(C,P)$ along $e$; by definition this is independent of choice of $(C',P')\in [[C,P]]_{e}$.
\end{definition}
These constructions are designed so that we can `splice' the local cutsets along each edge. Though this definition is somewhat complicated, note the following remark.

\begin{remark}
Let $e,v,w,C_{v},C_{w}$ be as above. If both $C_{v},C_{w}$ are proper with canonical partitions $P_{v},P_{w}$, then $(C_{v},P_{v})\sim_{e}(C_{w},P_{w})$. 
\end{remark}

\begin{definition}[Evenly \emph{$\sigma$-separated}]
Let $X$ be a non positively curved polygonal complex. We call $X$ \emph{evenly edge $\sigma$-separated} (respectively \emph{evenly vertex $\sigma$-separated}) if, letting $\Delta=\Delta_{X}$ (respectively $\Delta=X^{(1)}$):
\begin{enumerate}[label=$\roman*)$]
\item $X$ is regular (respectively $X$ is allowed \textbf{not to be regular})
    \item the link of every vertex in $X$ is edge (respectively vertex) $\sigma$-separated,
    \item for every vertex $v$ of $X$ and every $\sigma$-separated cutset $C$ in $Lk(v)$ there exists a series of partitions $\{P_{i}(C)\}_{i}$ of $\pi_{0}(Lk(v)-C)$ such that for any distinct pair of points $x,y\in Lk(v)$ separated by $C$, $x$ and $y$ are separated by some $P_{i}(C),$  and
    \item there exists an $Aut(X)$-invariant, strictly positive integer solution to the \emph{gluing equations}: we can assign a positive integer $\mu (C,P)$ to every pair $$(C,P)\in\mathcal{CP}:=\bigcup\limits _{e\in E^{\pm}(\Delta )}\mathcal{CP}(e)$$ such that 
    \begin{itemize}
        \item $\mu (C,P)$ is invariant under the action of $Aut(X)$, i.e. for every vertex $v$ of $X$, every $(C,P)\in\mathcal{CP}_{v}$ and $g\in Aut(X)$, we have that $g(C,P)\in\mathcal{CP}_{gv}$ with $\mu(C,P)=\mu (gC,gP)$,
        \item for every edge $e$ of $\Delta$ and every $(C,P)\in \mathcal{CP}(e)$,
\begin{equation*}
	    \sum\limits_{(C',P')\in [[C,P]]_{e}} \mu(C',P')=\sum\limits_{(C',P')\in [C,P]_{e^{-1}}} \mu(C',P'),
	\end{equation*} and
	\item for any $(C,P)\in\mathcal{CP}$ and any edges $e,e'\in C$,
	\begin{equation*}
	    \sum\limits_{(C',P')\in [[C,P]]_{e}} \mu(C',P')=\sum\limits_{(C',P')\in [[C,P]]_{e'}} \mu(C',P').
	\end{equation*}
	\end{itemize}
\end{enumerate}
We call $X$ \emph{evenly $\sigma$-separated} if it is either evenly vertex or edge $\sigma$-separated.
\end{definition}

As we will see, these cutsets and integer weights may be found by computer search; this is simplified greatly if $Aut(\Gamma)$ acts transitively on edges or vertices (as appropriate).

\section{Spaces with measured walls and the Haagerup property}\label{measured wallspaces}

In order to prove Theorem \ref{mainthm: link conditions for the Haagerup property}, we turn to the definition of a space with measured walls, as introduced in \cite{Cheriz-Martin-Valette_haagerupproperty}. In our setting we will restrict to considering only proper metric spaces, i.e. ones in which closed balls are compact. However, the results of \cite{Cheriz-Martin-Valette_haagerupproperty} are far more general.
\begin{definition}[\emph{Spaces with measured walls}]
    Let $(X,d)$ be a proper metric space. A \emph{wall} is a pair $W=\{U,V\}$, with $X=U\sqcup V$ a partition of $X$. The \emph{halfspaces} of $W$ are the sets $U$ and $V$. We say a wall \emph{separates} two points $u$ and $v$ if $u$ lies in $U$ and $v$ lies in $V$, or vice versa.
    
 A \emph{space with measured walls} is a tuple $((X,d),\mathcal{W},\mathcal{B},\nu)$, where $(X,d)$ is a proper metric space, $\mathcal{W}$ is a collection of walls in $X$, $\mathcal{B}$ is a $\sigma$-algebra on $\mathcal{W}$, $(\mathcal{W},\mathcal{B},\nu)$ is a measure space, and for any points $x,y\in X$, the set, $\omega (x,y)$, of walls separating $x$ and $y$ lies in $\mathcal{B}$, with $\nu(\omega (x,y))<\infty .$
\end{definition}
Given a space with measured walls, the map $d_{\nu}:(x,y)\mapsto \nu(\omega (x,y))$ is a pseudo-metric. We say that a group $G$ acts \emph{properly} on the space with measured walls if $G$ acts properly on $(X,d_{\nu}).$ The use of this is the following.
\begin{theorem*}\cite[Proposition 1]{Cheriz-Martin-Valette_haagerupproperty}
Let $G$ be a locally-compact group acting properly on a space with measured walls $(X,\mathcal{W},\mathcal{B},\nu)$. Then $G$ has the Haagerup property.
\end{theorem*}

\begin{remark}
We note that there could be some confusion regarding the definition of a group $G$ acting properly on $(X,d_{\nu})$. Indeed, since the space $(X,d_{\nu})$ will not necessarily be Hausdorff, the differing definitions of $G$ acting properly on $(X,d_{\nu})$ need not coincide. Of course all definitions of a proper action are equivalent for the action of $G$ on $(X,d)$, since $X$ is proper and Hausdorff. 

For our purpose, the definition of \emph{Palais proper} will be most useful, i.e. for any $x\in X$ there exists $U_{x}\ni x$ an open neighbourhood such that for any $y\in X$ there exists $V_{y}\ni y$ an open neighbourhood so that the subgroup $$G(U_{x}\;\vert\; V_{y}):=\{g\in G\;:\;gU_{x}\cap V_{y}\neq \emptyset\}$$ has compact closure. The other definition that will be useful is that of \emph{Borel proper}; for any compact subsets $K,L\subseteq X$ the group $G(K\;\vert\;L)$ is compact.

Indeed, any of the definitions of a proper action of $G$ on $(X,d_{\nu})$ will suffice, as follows. The main idea of \cite{Cheriz-Martin-Valette_haagerupproperty} is that the function $\psi:G\rightarrow \mathbb{R}_{\geq 0}$ which maps, for some fixed choice of $x\in X$, $g\mapsto d_{\nu}(x,gx)$, is continuous and \emph{conditionally negative definite}. By \cite{Akemann_Martin_Unbounded}, $G$ has the Haagerup property if $\psi$ is unbounded, i.e. $\lim_{g\rightarrow \infty}\psi(g)=\infty$. Of course, any reasonable definition of a proper action of $G$ on $(X,d_{\nu})$ will guarantee this.
\end{remark}

\begin{definition}
We say a space with measured walls $((X,d),\mathcal{W},\mathcal{B},\nu)$ is \emph{$f$-separated} if there exists an injective non-decreasing function $f:\mathbb{R}_{\geq 0}\rightarrow \mathbb{R}_{\geq 0}$ such that for any $x,y\in X$:
$$d(x,y)\leq f(d_{\nu}(x,y)).$$
\end{definition}
Using this, we can prove the following. Recall that for $K,L\subseteq X$, we let $G(K\;\vert\; L):=\{g\in G\;:\;gK\cap L\neq \emptyset\}.$

\begin{lemma}\label{lem: measured wallspace f-separated implies Haagerup}
Let $(X,d)$ be a proper metric space and $((X,d),\mathcal{W},\mathcal{B},\nu)$ a space with measured walls that is $f$-separated. Suppose that: $G$ acts on $(X,d)$ properly; $\nu$ is $G$-invariant; and $G$ acts continuously on the pseudo-metric space $(X,d_{\nu})$. Then $G$ has the Haagerup property.
\end{lemma}
\begin{proof}
We only need to show that the action on the pseudo-metric space $(X,d_{\nu})$ is Palais proper. 
For $x\in X$ and $r>0$, let $B(x,r)$ be the open ball around $x$ of distance $r$ under $d$.
Since $((X,d),\mathcal{W},\mathcal{B},\nu ) $ is $f$-separated, for any   points $x,y\in X$; $f(d_{\nu}(x,y))\geq d(x,y). $
Therefore, we see that for any $x\in X$ and $r>0$, $B_{d_{\nu}}(x,r)\subseteq B(x,f^{-1}(r)).$ Hence, for any $r,R>0$ and $x,y\in X$:
\begin{align*}
    G(B_{d_{\nu}}(x,r)\;\vert\; B_{d_
{\nu}}(y,R))&\leq G(B(x,f^{-1}(r))\;\vert\; B(y,f^{-1}(R))))\\
&\leq G(\overline{B(x,f^{-1}(r))}\;\vert\; \overline{B(y,f^{-1}(R))}).
\end{align*}
Since $X$ is a proper metric space and $G$ acts properly on $(X,d)$, the subgroup $$G(\overline{B(x,f^{-1}(r))}\;\vert\; \overline{B(y,f^{-1}(R))})$$ is compact, meaning that $G(B_{d_{\nu}}(x,r)\;\vert\; B_{d_{\nu}}(y,R))$ has compact closure. Therefore, $G$ acts on $(X,d_{\nu})$ Palais properly, and the result follows by \cite[Proposition 1]{Cheriz-Martin-Valette_haagerupproperty}.\end{proof}

\section{Aperiodic hypergraphs in \texorpdfstring{$\mathbf{\pi}$}{pi}-separated polygonal complexes and spaces with measured walls}\label{section: Constructing hypergraphs in polygonal complexes}In this section we use our local conditions to build a space with measured walls on which the group $Aut(X)$ acts.
\subsection{Aperiodic hypergraphs}
We now begin to construct our separating subcomplexes, and then define the elements of our set of walls.

\begin{definition}[Hypergraph stars and aperiodic hypergraphs]
Let $X$ be a CAT(0) polygonal complex satisfying the conditions of Theorem \ref{mainthm: link conditions for the Haagerup property}, and let $\Delta=\Delta_{X}$ if we are considering edge cutsets and $\Delta=X^{(1)}$ if we are dealing with vertex cutsets. Let $v$ be a vertex of $X$ and $C\subseteq E(Lk_{X}(v))$ (respectively $C\subseteq V(Lk_{X}(v))$) a $\pi$-separated cutset. Let $C'$ be the set of closed edges in $\Delta$ corresponding to $C$. The \emph{hypergraph star map at $v$}, $\phi_{hyp,v}:E(Lk_{X}(v))\rightarrow X$ (respectively $\phi_{hyp,v}:V(Lk_{X}(v))\rightarrow X$), is the map taking $C$ to the image of $C'$ under the map $\Delta\rightarrow X.$

We have assumed that $X$ satisfies the condition of Theorem \ref{mainthm: link conditions for the Haagerup property}. For each vertex $v$, let $\{(C,P_{C,i})\}_{C,i}$ be the required set of cutsets and associated partitions. Recall that we let
$$\mathcal{CP}_{v}=\{(C,P_{C,i})\;:\;C\mbox{ is a }\pi\mbox{-separated cutset},\;C\subseteq Lk(v),\;\vert C\vert\geq 2\}.$$ For $e$ an edge we defined
$\mathcal{CP}(e):=\{(C,P)\in \mathcal{CP}_{\iota (e)}\;:\;e\in C\},$
and  $$\mathcal{CP}=\bigcup_{e\in E^{\pm }(\Delta)}\mathcal{CP}(e).$$
By assumption $\mathcal{CP}_{v}$ is non-empty for all $v\in X^{(0)}$.
Let $\mathcal{CP}_{v}':=\mathcal{CP}_{v}\sqcup \{(\emptyset,\emptyset )\}.$
We define the space
$$\mathcal{S}:=\prod\limits_{v\in V(X)}\mathcal{CP}_{v}',$$
and the subspace 
$$\mathcal{T}:=\left\{(C_{v},P_{v})\}_{v}\in \mathcal{S}\;:\;\forall\;e\in E(\Delta)\mbox{ either }\; \begin{matrix}
(C_{\iota(e)},P_{\iota(e)})\sim_{e}(C_{\tau(e)},P_{\tau(e)})\mbox{, or}\\
e\notin C_{\iota(e)},\;e^{-1}\notin C_{\tau(e)}
\end{matrix}
\right\}.$$

The \emph{trivial hypergraph sequence} is the sequence $\vect{\epsilon}:=(\emptyset,\emptyset )_{v}$. We further define
$$\mathcal{H}=\mathcal{T}-\vect{\epsilon}.$$

The map $\phi_{hyp,v}$ on each vertex extends to a map $\phi_{hyp}:\mathcal{T}\rightarrow X$. An \emph{aperiodic hypergraph} is the image $\phi_{hyp}(\vect{x})$ of an element $\vect{x}\in\mathcal{H}.$ We typically write $\Lambda_{\vect{x}}:=\phi_{hyp}(\vect{x}).$ We say $\Lambda_{\vect{x}}$ \emph{passes through} the pair $\vect{x}_{v}$ at the vertex $v$. We call $\Lambda$ an \emph{aperiodic edge hypergraph} if it is composed of edge cutsets, and an \emph{aperiodic vertex hypergraph} otherwise.
\end{definition}
We call these hypergraphs aperiodic to distinguish them from the complexes constructed in \cite{Ashcroft2020LinkCF}, the key distinction between these is that the latter have subgroups $H\leq Aut(X)$ acting properly discontinuously and cocompactly on them.

It is easily seen that $\mathcal{H}$ is non-empty, and furthermore, for each $(C,P)\in \mathcal{CP}_{v}$ there exists $\vect{x}\in\mathcal{H}$ with $\vect{x}_{v}=(C,P).$ Importantly, for each vertex $v$ and for each pair $(C,P)\in \mathcal{CP}_{v}$, there exists an aperiodic hypergraph passing through $(C,P)$. We now note that hypergraphs are leafless trees.
\begin{lemma}\label{lem: edge hypergraphs are leafless convex}
	Let $X$ be a polygonal complex satisfying the conditions of Theorem \ref{mainthm: link conditions for the Haagerup property}, and let $\Lambda$ be an aperiodic hypergraph in $X$. Then $\Lambda$ is a leafless convex tree.
\end{lemma}

\begin{proof}
	As the cutsets are $\pi$-separated, $\Lambda$ is locally geodesic in $X$. As $X$ is CAT(0), local geodesics are geodesic, and geodesics are unique, so that $\Lambda$ is a convex tree.
	Since $\vert C\vert\geq 2$ for any $v\in V(\Delta)$ and $(C,P)\in\mathcal{CP}_{v}$, $\Lambda$ contains no vertices of degree $1$, and hence is leafless.
\end{proof}
	\begin{definition}
		 Let $\Lambda$ be an aperiodic hypergraph in $X$ and $x,y\in X$ be distinct points in $X$. We say $\Lambda$ \emph{separates} $x$ and $y$ if $x$ and $y$ lie in distinct components of $X-\Lambda$.
	\end{definition}
	
	We now consider separating points; we note the following lemma.
	We call a path $\gamma$ \emph{transverse} to $\Lambda$ if $\vert \gamma\cap \Lambda\vert =1$. If $x$ is a point on an edge $e$ of $\Lambda$, then there is a canonical partition of $Lk(x)-\Lambda$ obtained from the partitions of $Lk(v)-\Lambda$ and $Lk(w)-\Lambda$, where $v,w$ are the endpoints of $\Lambda$ (since these are equatable along $e$ the induced partitions are the same).
	\begin{lemma}\cite[Lemma 2.27]{Ashcroft2020LinkCF}\label{lem: separation condition}
	    Let $\Lambda$ be an aperiodic hypergraph in $X$, and $\gamma=[p,q]$ be a geodesic transverse to $\Lambda$. If $\gamma\cap\Lambda=\{x\}$ and $p$ and $q$ lie in different elements of the partition of $Lk(x)-\Lambda$, then $p$ and $q$ lie in different components of $X-\Lambda$.
	\end{lemma}

\subsection{The \texorpdfstring{$\boldsymbol{\sigma}$}{sigma}-algebra on walls}
We now need to show that we can use aperiodic hypergraphs to build a space with measured walls.

\begin{definition}[\emph{$\Lambda$ walls}]
    Let $\Lambda$ be an aperiodic hypergraph in $X$, with disjoint components $X-\Lambda=\{U^{i}_{\Lambda}\}_{i}$. For each $U_{\Lambda}^{i}$, let $V_{\Lambda}^{i}=X-\overline{U_{\Lambda}^{i}}$.
The \emph{set of $\Lambda$ walls} is the set

$$\mathcal{W}_{\Lambda}=\bigg\{\{\overline{U_{\Lambda}^{i}},V_{\Lambda}^{i}\}\;:\;U_{\Lambda}^{i}\mbox{ a component of }X-\Lambda\bigg\}.$$
The \emph{set of hypergraph walls} is the set of walls 
$$\mathcal{W}=\bigsqcup\limits_{\vect{x}\in\mathcal{H}}\mathcal{W}_{\Lambda_{\vect{x}}},$$
where we \emph{do not} remove any duplicate walls, i.e. we mark each wall with the hypergraph it corresponds to. Note the trivial hypergraph does not appear in the above union.
\end{definition}

We do however need to add the trivial hypergraph, so we add $$\mathcal{V}=\mathcal{W}\sqcup \{\{X,\emptyset\}\}.$$

\subsubsection{The \texorpdfstring{$\boldsymbol{\sigma}$}{sigma}-algebra on \texorpdfstring{$\boldsymbol{\mathcal{W}}$}{W}}
 We now to describe the $\sigma$-algebra $\mathcal{B}$ on $\mathcal{W}$. 
Let $$\mathcal{A}:=\bigsqcup_{\vect{x}\in \mathcal{H}}\bigsqcup\limits_{U^{i}_{\Lambda_{\vect{x}}}}(\vect{x},\{\overline{U^{i}_{\Lambda_{\vect{x}}}},V^{i}_{\Lambda_{\vect{x}}}\}).$$

There is a bijective map $\psi:\mathcal{A}\rightarrow\mathcal{W},$ and so we will identify these two sets. 
Since $X$ has finitely many shapes, choose $\kappa >0 $ such that the length of any edge is at least $1000\kappa$.
Given a finite set of vertices $v_{1},\hdots , v_{n}$ and a hypergraph wall $W=\{\overline{U_{\Lambda}},V_{\Lambda}\}$, where $\Lambda\cap \{v_{1},\hdots ,v_{n}\}\neq \emptyset$, define the \emph{cylinder set} $\mathcal{A}(v_{1},\hdots v_{n},\{\overline{U_{\Lambda}},V_{\Lambda}\})$
as the set
\begin{align*}
\left\{(\vect{x},\{\overline{U_{\Lambda_{\vect{x}}}},V_{\Lambda_{\vect{x}}}\})\in \mathcal{A}\;:\;\hspace*{-5 pt}\begin{array}{l l}\Lambda_{\vect{x}}\cap \bigg(\bigcup\limits_{i} \mathcal{N}_{\kappa}(v_{i})\bigg)=\Lambda\cap \bigg(\bigcup\limits_{i} \mathcal{N}_{\kappa}(v_{i})\bigg),\\
\overline{U_{\Lambda_{\vect{x}}}}\cap \bigg(\bigcup\limits_{i} \mathcal{N}_{\kappa}(v_{i})\bigg)=\overline{U_{\Lambda}}\cap \bigg(\bigcup\limits_{i} \mathcal{N}_{\kappa}(v_{i})\bigg)
\end{array}\right\}
\end{align*}
Let $\mathcal{B}$ be the $\sigma$-algebra on $\mathcal{W}$ generated by the cylinder sets.

Let $x,y\in X$ and let $\gamma$ be a geodesic between them. A wall $\{\overline{U_{\Lambda}},V_{\Lambda}\}$ separates $x$ and $y$ iff $\gamma\cap \Lambda$ is connected, and $x$ and $y$ lie in different elements of $\{\overline{U_{\Lambda}},V_{\Lambda}\}$. Therefore, as $\gamma$ has finite length, the set, $\omega (x,y)$, of walls separating $x$ and $y$ is a finite union of cylinder sets, which lies in $\mathcal{W}$, i.e. $\omega (x,y)\in\mathcal{W}$.

\subsubsection{The \texorpdfstring{$\boldsymbol{\sigma}$}{sigma}-algebra on \texorpdfstring{$\boldsymbol{\mathcal{V}}$}{V}}
We define the $\sigma$-algebra $\mathcal{B}'$ on $\mathcal{V}$.
Let $$\mathcal{A}':=\mathcal{A}\sqcup \{\vect{\epsilon},\{X_{\vect{\epsilon}},\emptyset_{\vect{\epsilon}}\}\}.$$
There is a bijective map $\psi:\mathcal{A}'\rightarrow\mathcal{V}$, and so we will identify these two sets. 
Define the cyclinder set,  $\mathcal{A}'(v_{1},\hdots v_{n},\{\overline{U_{\Lambda}},V_{\Lambda}\}),$ for $\Lambda\cap(\cup_{i}v_{i})\neq \emptyset$
as
\begin{align*}
\left\{(\vect{x},\{\overline{U_{\Lambda_{\vect{x}}}},V_{\Lambda_{\vect{x}}}\})\in \mathcal{A'}\;:\;\hspace*{-5 pt}\begin{array}{l l}\Lambda_{\vect{x}}\cap \bigg(\bigcup\limits_{i} \mathcal{N}_{\kappa}(v_{i})\bigg)=\Lambda\cap \bigg(\bigcup\limits_{i} \mathcal{N}_{\kappa}(v_{i})\bigg),\\
\overline{U_{\Lambda_{\vect{x}}}}\cap \bigg(\bigcup\limits_{i} \mathcal{N}_{\kappa}(v_{i})\bigg)=\overline{U_{\Lambda}}\cap \bigg(\bigcup\limits_{i} \mathcal{N}_{\kappa}(v_{i})\bigg)
\end{array}\right\}.
\end{align*}
If  $\Lambda\cap(\cup_{i} v_{i})=\emptyset,$ then define $\mathcal{A}'(v_{1},\hdots ,v_{n},\{\overline{U_{\Lambda}},V_{\Lambda}\})$ as 
\begin{align*}
\left\{(\vect{x},\{\overline{U_{\Lambda_{\vect{x}}}},V_{\Lambda_{\vect{x}}}\})\in \mathcal{A'}\;:\;\hspace*{-5 pt}\begin{array}{l l}\Lambda_{\vect{x}}\cap \bigg(\bigcup\limits_{i} \mathcal{N}_{\kappa}(v_{i})\bigg)=\emptyset,\\
\overline{U_{\Lambda_{\vect{x}}}}\cap \bigg(\bigcup\limits_{i} \mathcal{N}_{\kappa}(v_{i})\bigg)=\overline{U_{\Lambda}}\cap \bigg(\bigcup\limits_{i} \mathcal{N}_{\kappa}(v_{i})\bigg)
\end{array}\right\}.
\end{align*}
Let $\mathcal{B}'$ be the $\sigma$-algebra generated by cyclinder sets. 

\subsection{The measure on the space of walls}

Now, we note that $\mathcal{T}$ is a closed subset of $\mathcal{S}$ and hence is compact. It follows that the space $\mathcal{A}'$ is also compact.

For a hypergraph $\Lambda$, and a vertex $v$, let $\Lambda_{v}$ be the pair $(C,P)$ through which $\Lambda$ passes in $Lk(v)$ (this is allowed to be empty, i.e. $(\emptyset,\emptyset)$).
If $C\neq\emptyset$, define $$M(C,P)=\sum\limits_{(C',P')\in [[C,P]]_{e}} \mu(C',P'),$$ where $e$ is any edge in $C$ (by definition, this does not depend on the choice of $e$), and otherwise let $M(\emptyset,\emptyset):=1.$

\subsubsection{The measure \texorpdfstring{$\boldsymbol{\mu}$}{mu} on the space  \texorpdfstring{$\boldsymbol{(\mathcal{A'},\mathcal{B}')}$}{(A',B')}}
For a cylinder set $\mathcal{A}'(v_{1},\hdots ,v_{n},\{\overline{U_{\Lambda}},V_{\Lambda}\}),$ with $\Lambda \cap v_{i}\neq \emptyset$ for some $v_{i}$ define
$$\mu(\mathcal{A}'(v_{1},\hdots ,v_{n},\{\overline{U_{\Lambda}},V_{\Lambda}\})):=\dfrac{1}{M(\Lambda_{v_{1}})\hdots M(\Lambda_{v_{n}})}.$$

We now show that $\mu$ is a measure on $(\mathcal{A}',\mathcal{B}')$.
\begin{lemma}
The set function $\mu$ is a measure.
\end{lemma}
\begin{proof}
Choose a finite sequence of sets $V_{n}\subseteq X^{(0)}$ with $\cup_{n}V_{n}=X^{(0)}$. Define the space $\mathcal{A}_{n}'$ as the quotient of $\mathcal{A}'$ under the map induced by $V_{n}\rightarrow V$. That is,
\begin{align*}
    \mathcal{A}_{n}'&=\left\{(\vect{x},\{\overline{U},V\})\;:\begin{matrix}&\vect{x}\in \prod_{v\in V_{n}}\mathcal{CP}_{v}\mbox{ such that }\exists\; \vect{y}\in\mathcal{A}'\mbox{ with}\\
&\vect{x}_{v}=\vect{y}_{v}\;\forall v\in V,\; 
\{\overline{U}_{\Lambda_{\vect{y}}},V_{\Lambda_{\vect{y}}}\}=\{\overline{U},V\}\end{matrix}\right\}.
\end{align*}

There are associated projection maps $p_{n}:\mathcal{A}'\rightarrow \mathcal{A}_{n}'$ and $\rho _{m,n}:\mathcal{A}_{m}'\rightarrow \mathcal{A}_{n}'$ for any $m>n$. The above construction of $\mu$ also descends to a construction of a measure $\mu_{n}$ on each $\mathcal{A}_{n}'$ such that $p_{n}\mu=\mu_{n}$, and $\rho_{m,n}\mu_{m}=\mu_{n}$. If $\mu (\mathcal{A}')<\infty,$ then it follows by a classical application of the Carath\'{e}odory extension theorem that $\mu$ is a measure on $\mathcal{B}'$ (see e.g. \cite[Corollary 14.36]{Klenke2020}). 

Otherwise, $\mu (\mathcal{A}')=\infty,$ and so $\mu(\mathcal{A}'(v_{1},\hdots, v_{n},(X,\emptyset)))=\infty$ for all finite sets $\{v_{1},\hdots ,v_{n}\}$. We claim there are infinitely many disjoint cylinder sets $A_{i}:=\mathcal{A}'(v_{i_{1}},\hdots ,v_{m_{i}},\{U_{\Lambda_{i}}V_{\Lambda_{i}}\})$, with  $\Lambda_{i}\cap (\cup_{j}v_{i_{j}})\neq \emptyset$ (so that $\mu (A_{i})<\infty)$, and $\sum\limits_{i}\mu (A_{i})=\infty$. If not, then $\mu(\mathcal{A}')<\infty$.

Let $\mathcal{R}$ be the semiring of cyclinder sets in $\mathcal{A}$; this semiring generates the $\sigma$-algebra $\mathcal{B}$, and so we aim to apply the Carath\'eodry extension theorem. It is clear that $\mu$ is finitely additive and $Aut(X)$-invariant, since $X$ is evenly $\pi$-separated, and so the gluing equations are satisfied.

Now, let $C_{n}$ be a pairwise disjoint sequence of sets such that $C=\cup_{n}C_{n}\in \mathcal{R}.$ We need to show that $\mu (C)=\sum_{n}\mu (C_{n}).$ If $\mu(C)$ is finite, then this follows immediately by application of classical techniques that appear e.g. in the proof of the Kolmogorov extension theorem, as follows. Let $D_{N}=\sqcup _{n=1}^{N}C_{n}$ and $E_{N}=C- D_{N}.$ Then 
$$\mu (C)=\mu (D_{N})+\mu (E_{N})=\sum_{n=1}^{N}\mu (C_{n})+\mu(E_{n}).$$ It suffices to show that $\mu(E_{N})\rightarrow 0.$ Suppose otherwise, so that $\mu(E_{N})\rightarrow\epsilon$ for some $\epsilon>0$. We note that $\cap_{N} E_{N}=0$. We may choose for each $M,N$ a compact set $K_{M,N}$, as follows. Since $\mu_{m} (\rho _{m}E_{N})\rightarrow \mu(E_{N})$ as $m\rightarrow\infty$, for each $N$ we may choose $M_{N}$ such that $\vert \mu(E_{N})-\mu_{M}(E_{N})\vert \leq \epsilon\slash 1000$ for all $M\geq M_{N}$. Furthermore, we may require that $M_{N+1}\geq M_{N}$. Let $K_{M,N}:=\rho(M)^{-1}(\rho_{M}(E_{N})).$ Then $K_{M, N+1}\subseteq K_{M,N}$ and $K_{M+1,N}\subseteq K_{M,N}$. Furthermore, each $K_{M,N}$ is compact and non empty. Let $K=\cap_{N}\cap_{M=M_{N}}^{\infty}K_{M,N}$, so that $K$ is non empty. However, if $k\in K$, then for all $N$ and $M\geq M_{N}$, $k\in \rho_{M}^{-1}(E_{N})$. Therefore, $k\in E_{N}$ for all $N$, and hence $k\in \cap_{N}E_{N}=\emptyset$, a contradiction. Therefore $\mu(E_{N})\rightarrow0,$ and $\mu(C)=\sum_{n}\mu (C_{n}).$

If $\mu(C)=\infty,$ then we must have that $$
C=\mathcal{A}(v_{1},\hdots ,v_{m},(X,\emptyset));$$
this is the only possible cylinder set with infinite measure.  We may assume that $\mu (C_{n})<\infty$ for each $n$, otherwise we are finished. Note that $\mu(X-C)<\infty. $ Therefore, by applying the previous paragraph, we see that $$\mu (X-C)=\sum_{i}\mu (A_{i}\cap (X-C)) <\infty ,$$ and so $$\sum\mu (A_{i}\cap C)=\sum_{i} [\mu (A_{i})-\mu(A_{i}-C)]=\mu(X)-\mu(X-C)=\infty.$$

Let us replace the sequence $C_{n}$ by the sequence $A_{i}\cap C_{n}.$ Since $\mu (C_{n})<\infty $ for each $n$, we may also apply the previous paragraph to deduce that for all $n$:
$$\mu(C_{n})=\sum_{i}\mu (A_{i}\cap C_{n}).$$
Since $\mu (A_{i})<\infty$ for each $i$, we also have that for all $i$: $$\mu (A_{i}\cap C)=\sum_{n}\mu (A_{i}\cap C_{n}).$$
Therefore: 
\begin{align*}
  \sum_{n}\mu(C_{n})&=\sum_{n}\left(\sum_{i}\mu (A_{i}\cap C_{n}) \right)\\
  &=  \sum_{i}\left(\sum_{n}\mu (A_{i}\cap C_{n}) \right)\\
  &=\sum_{i}\mu (A_{i}\cap C)\\&=\infty.
\end{align*}

It follows that $\mu$ is a countably-additive measure on $\mathcal{R}$, and so, by the Carath\'{e}odory extension theorem, extends to a measure on the $\sigma$-algebra generated by $\mathcal{R}$,  $\mathcal{B}'$.
\end{proof}

\subsubsection{The measure \texorpdfstring{$\boldsymbol{\nu}$}{nu} on the space  \texorpdfstring{$\boldsymbol{(\mathcal{A},\mathcal{B})}$}{(A,B)}}
Take $\nu:=\mu\vert_{\mathcal{A}}.$ This is a measure on $\mathcal{A}$. Note that there exists a constant $c>0$ such that for any triple of vertices $v_{1},v_{2},v_{3}$ connected by edges in $\Lambda$:

$$\nu (\mathcal{A}(v_{1},v_{2},v_{3},\{\overline{U_{\Lambda}},V_{\Lambda}\}))\geq c.$$
It now remains to show that the space is a space with measured walls.
\begin{lemma}\label{lem: hypergraph wallspace is measured}
The space $(X,\mathcal{W},\mathcal{B},\nu)$ is a space with measured walls on which $Aut(X)$ acts continuously.
\end{lemma} 
\begin{proof}
We first need to prove that for any $x,y\in X$, $\nu (\omega (x,y))<\infty$. Let $\gamma$ be a geodesic from $x$ to $y$. Since there are finitely many types of links in $X$ and $Shapes(X)$ is finite, we see that there is an upper bound $K$ such that any path of length $1$ meets at most $K$ hypergraph edges (there may, however, be infinitely many hypergraphs containing this edge). The only time that a wall can separate $x$ and $y$ is when $\gamma\cap \Lambda$ is connected, and locally $\gamma$ lies in different elements of the partition of $X-\Lambda$. There are at most $2$ choices of partition of $X-\Lambda$ that satisfy this; $\overline{U_{\Lambda}}$ must contain exactly one of $x$ or $y$. Letting $L$ be the maximal number of cutsets in a link, we see that
\begin{align*}
    \nu (\omega (x,y))&\leq 2 KL^{2}(d(x,y)+1)\max\left\{\nu (\Sigma)\;:\;\begin{matrix}
    \Sigma=\mathcal{A}(v_{1},\hdots ,v_{n},\{\overline{U_{\Lambda}},V_{\Lambda}\}),\\
    \Lambda\cap(\cup_{i}v_{i})\neq \emptyset\end{matrix} \right\}\\&\leq 2KL^{2}(d(x,y)+1).
\end{align*}

We now need to show that $Aut(X)$ acts on $(X,d_{\nu})$ continuously. Since $\nu$ is $Aut(X)$-invariant, to prove this it suffices to show that if $g_{n}\rightarrow g$, then $d_{\nu}(g_{n}x,y)\rightarrow d_{\nu}(gx,y)$ for $x,y\in X$. Suppose $x,y \in X$, and $\{g_{n}\}_{n}\subseteq Aut(X)$ with $g_{n}\rightarrow g$. There are three cases. We note that since $Aut(X)$ acts on $X$, it sends vertices to vertices, edges to edges, and polygonal faces to polygonal faces. Firstly if $x$ is a vertex, then there exists $N\geq 1$ such that for all $n\geq N$ $g_{n}x=gx$. If $x$ lies on an edge, then we must have, again, that $g_{n}x$ lies on the same edge as $gx$ for all $n$ suitably large. Finally, if $x$ lies in the interior of a $2$-cell, then $g_{n}x$ and $gx$ lie in the interior of the same $2$-cell for all $n$ sufficiently large.

In all of the above cases, we may conclude immediately that $d_{\nu}(gx,y)=\lim_{n}d_{\nu}(g_{n}x,y).$ Hence, $Aut(X)$ acts continuously on $(X,d_{\nu})$.
\end{proof}

\section{The Haagerup property for groups acting on polygonal complexes}\label{section: Haagerup property for groups acting on polygonal complexes}	
We now wish to prove there exists a proper action on our space with measured walls; we do this by showing the space is $f$-separated. For a metric polygonal complex $X$, let $D(X)$ be the maximal circumference of a polygonal face in $X$. Since $X$ has finitely many shapes, $D(X)$ is finite. The following result is essentially proved in \cite{Ashcroft2020LinkCF}; we provide the proof for completeness.
	\begin{lemma}\label{lem: hypergraph separation}\cite[Lemma 2.42,\;2.43]{Ashcroft2020LinkCF}
		Let $X$ be a CAT(0) polygonal complex as in the statement of Theorem \ref{mainthm: link conditions for the Haagerup property}. Let $\gamma$ be a finite geodesic in $X$ of length at least $4D(X)$. There exists an aperiodic hypergraph $\Lambda$, an associated wall $\{\overline{U_{\Lambda}},V_{\Lambda}\}$ and a triple of vertices $(v_{1},v_{2},v_{3})$ such that any wall in $\mathcal{\mathcal{A}}(v_{1},v_{2},v_{3},\{\overline{U_{\Lambda}},V_{\Lambda}\})$ separates the endpoints of any finite geodesic extension of $\gamma$.
	\end{lemma}
	\begin{proof}

Let us suppose that $X$ is evenly $\pi$-separated with vertex cutsets. Since $\gamma$ is of length at least $4D(X)$, we can write $\gamma=\gamma_{1}\cdot\delta\cdot\gamma_{2}$, where each $\gamma_{i}$ is of length at least $D(X)\slash 2$, and $\delta$ is a path of length between $D(X)$ and $2D(X)$ that starts at a point $v\in X^{(1)}$ and ends at $w\in X^{(1)}$. 
		
	First suppose that  $\delta$ contains a nontrivial subpath, $\delta'$, which contains exactly one point of $X^{(1)}$, $u$, in its interior. Let $e$ be the edge of $X$ containing $u$. Since $\delta'$ is geodesic, we see that $\iota (\delta')$ and $\tau (\delta')$ lie in two distinct faces; as $X$ is evenly vertex $\pi$-separated there exists a $\pi$-separated cutset $C\ni e$ and partition $P$ of $\pi_{0}(Lk(\iota (e))-C)$ with $\iota (\delta '),\tau (\delta')$ lying in distinct elements of $P$. Let $\Lambda$ be any hypergraph passing through $(C,P)$ in $Lk(\iota (e))$; by Lemma \ref{lem: separation condition} $\Lambda$ separates the endpoints of $\delta'$, and hence the endpoints of $\gamma$.
		
		Otherwise, $\delta$ contains a subpath that is contained entirely in $X^{(1)}$; therefore $\delta$ must meet a vertex $v$ of $X$. Let $\delta_{1}$, $\delta_{2}$ be the two subpaths of $\gamma$ incident to $v$; as $\gamma$ is geodesic, $d_{Lk(v)}(\delta_{1},\delta_{2})\geq \pi$. Let $C$ be a vertex cutset such that $\gamma_{1}$ and $\gamma_{2}$ lie in different components of $Lk(v)-C$ and let $P$ be a chosen partition of $\pi_{0}(Lk(v)-C)$ separating $\gamma_{1}$ and $\gamma_{2}$ (this exists as $X$ is evenly vertex $\pi$-separated). Let $\Lambda$ be any vertex hypergraph passing through $(C,P)$ in $Lk(v)$; by Lemma \ref{lem: separation condition} this separates $\gamma_{1}$ and $\gamma_{2}$, and so separates the endpoints of $\gamma$. Again, choosing the appropriate wall, the result follows. 
		
		The proof is similar for edge cutsets.
	\end{proof}
	
	\begin{lemma}\label{lem: relating wall distance and distance}
	There exists a constant $\lambda\geq 1$ such that for any points $x,y\in X:$
	$$(d(x,y)-1)\slash \lambda \leq \nu (\omega (x,y))\leq \lambda(d(x,y)+1).$$
	\end{lemma}
	In particular, the space with measured walls is $f$-separated for $f$ a linear function, i.e. $f(x)=\lambda x+ \lambda $.
\begin{proof}
We already know that there exists a constant $\lambda'>0$ such that for any $x,y$: $\nu (\omega (x,y))\leq \lambda' (d(x,y)+1)$. By Lemma \ref{lem: hypergraph separation}, we see that $$\nu (\omega (x,y))\geq \dfrac{c}{D(X)}(d(x,y)-1),$$
where $c$ is the minimal measure of a cylinder set of the form $\mathcal{\mathcal{A}}(v_{1},v_{2},v_{3},\{\overline{U_{\Lambda}},V_{\Lambda}\})$, which is non-zero.
\end{proof}

We may now immediately conclude our main theorem. 

\begin{proof}[Proof of Theorem \ref{mainthm: link conditions for the Haagerup property}]

Let $((X,d),\mathcal{W},\mathcal{B},\nu)$ be the space with measured walls constructed previously (we know it is a space with measured walls by Lemma \ref{lem: hypergraph wallspace is measured}). We know that, by definition, $Aut(X)$ acts properly on $(X,d)$. Furthermore, by Lemma \ref{lem: hypergraph wallspace is measured}, $Aut(X)$ acts continuously on the space $(X,d_{\nu})$. By Lemma \ref{lem: hypergraph separation} the space with measured walls is $f$-separated, and so Theorem \ref{mainthm: link conditions for the Haagerup property} follows by Lemma \ref{lem: measured wallspace f-separated implies Haagerup}.
\end{proof}

\section{Application to \'{S}wi\k{a}tkowski's complexes}\label{section: applications}
We now turn our attention to prove Corollary \ref{maincor: application to F090A complexes}; we begin by discussing cubic graphs.

\begin{definition}[Cubic graphs]
Let $\Gamma$ be a finite graph; we call $\Gamma$ \emph{cubic} if it is connected, bipartite, and trivalent.  
\end{definition}
If $\Gamma$ is cubic, for each $v\in V(\Gamma)$ we will pick an ordering of the neighbours of $v$, $w_{1}(v)$, $w_{2}(v)$, $w_{3}(v)$.
We introduce the notion of a $*$-separated graph, a strengthening of vertex $3$-separated, as it is easier to verify the gluing equations for complexes whose links are such graphs.
\begin{definition}[$*$-separated cutsets]
Let $C$ be a vertex cutset in a graph $\Gamma$. We say $C$ is a \emph{$*$-separated cutset} if $C$ is $3$-separated under the combinatorial metric, $\Gamma - C$ contains exactly two components, and $C$ is minimal (i.e. for any $w\in C$, $C-\{w\}$ is not a cutset).
\end{definition}

\begin{definition}
Let $\Gamma$ be a cubic graph, and let $\mathcal{C}$ be a collection of $*$-separated cutsets. For $v\in V(\Gamma)$, we define
$$*(v,i,j)$$
to be the set of all $*$-separated cutsets $C\ni v$ such that $w_{i}(v)$ and $w_{j}(v)$ lie in the same connected component of $\Gamma - C$. We further define $$\mathcal{C}(v,i,j):=\mathcal{C}\cap *(v,i,j).$$
\end{definition}
\begin{definition}[$*$-separated graph]
Let $\Gamma$ be a graph. We say that $\Gamma$ is \emph{$*$-separated} if:
\begin{enumerate}[label=$\roman*)$]
	\item $\Gamma$ is a cubic graph,
	\item $\Gamma$ is vertex $3$-separated by a set $\mathcal{C}$ of $*$-separated cutsets (and hence for any  vertex $v$ and any $i\neq j$, $\mathcal{C}(v,i,j)$ is non-empty), and
	\item there exists an integer $M$ such that for any vertex $v$ and any $i\neq j $, $\vert C\in\mathcal{C}(v,i,j)\vert  =M \slash 3.$
\end{enumerate}
\end{definition}
We prove the following lemma.

\begin{lemma}\label{lem: F090 separation}
Endow the graph $F090A$ with the combinatorial metric. Then $F090A$ is $*$-separated.
\end{lemma}
\begin{longtable}{c|ccl c c|ccl c c|ccl}
\caption{Edge incidences for $F090A$}\\
$v_{i}$  &\multicolumn{3}{c}{$v_{j}$ adjacent to $v_{i}$} &&	$v_{i}$  &\multicolumn{3}{c}{$v_{j}$ adjacent to $v_{i}$}& &	$v_{i}$ &\multicolumn{3}{c}{$v_{j}$ adjacent to $v_{i}$}\\
	\hline
1 & 2 & 18 & 90 && 31 & 30 & 32 & 48 && 61 & 60 & 62 & 78 \\
 2 & 1 & 3 & 83 && 32 & 23 & 31 & 33 && 62 & 53 & 61 & 63 \\
 3 & 2 & 4 & 40 && 33 & 32 & 34 & 70 && 63 & 10 & 62 & 64 \\
 4 & 3 & 5 & 57 &&34 & 33 & 35 & 87 && 64 & 27 & 63 & 65 \\
 5 & 4 & 6 & 14 &&35 & 34 & 36 & 44 && 65 & 64 & 66 & 74 \\
 6 & 5 & 7 & 79 && 36 & 19 & 35 & 37 && 66 & 49 & 65 & 67 \\
 7 & 6 & 8 & 24 && 37 & 36 & 38 & 54 && 67 & 66 & 68 & 84 \\
 8 & 7 & 9 & 89 && 38 & 29 & 37 & 39 && 68 & 59 & 67 & 69 \\
 9 & 8 & 10 & 46 && 39 & 38 & 40 & 76 && 69 & 16 & 68 & 70 \\
 10 & 9 & 11 & 63 && 40 & 3 & 39 & 41 && 70 & 33 & 69 & 71 \\
 11 & 10 & 12 & 20 &&41 & 40 & 42 & 50 && 71 & 70 & 72 & 80 \\
 12 & 11 & 13 & 85 &&42 & 25 & 41 & 43 && 72 & 55 & 71 & 73 \\
 13 & 12 & 14 & 30 &&43 & 42 & 44 & 60 && 73 & 72 & 74 & 90 \\
 14 & 5 & 13 & 15 &&44 & 35 & 43 & 45 && 74 & 65 & 73 & 75 \\
 15 & 14 & 16 & 52 &&45 & 44 & 46 & 82 && 75 & 22 & 74 & 76 \\
 16 & 15 & 17 & 69 && 46 & 9 & 45 & 47 && 76 & 39 & 75 & 77 \\
 17 & 16 & 18 & 26 && 47 & 46 & 48 & 56 && 77 & 76 & 78 & 86 \\
 18 & 1 & 17 & 19 &&48 & 31 & 47 & 49 && 78 & 61 & 77 & 79 \\
 19 & 18 & 20 & 36 && 49 & 48 & 50 & 66 && 79 & 6 & 78 & 80 \\
 20 & 11 & 19 & 21 && 50 & 41 & 49 & 51 && 80 & 71 & 79 & 81 \\
 21 & 20 & 22 & 58 &&51 & 50 & 52 & 88 && 81 & 28 & 80 & 82 \\
 22 & 21 & 23 & 75 && 52 & 15 & 51 & 53 && 82 & 45 & 81 & 83 \\
 23 & 22 & 24 & 32 &&53 & 52 & 54 & 62 && 83 & 2 & 82 & 84 \\
 24 & 7 & 23 & 25 &&54 & 37 & 53 & 55 && 84 & 67 & 83 & 85 \\
 25 & 24 & 26 & 42 && 55 & 54 & 56 & 72 && 85 & 12 & 84 & 86 \\
 26 & 17 & 25 & 27 && 56 & 47 & 55 & 57 && 86 & 77 & 85 & 87 \\
 27 & 26 & 28 & 64 && 57 & 4 & 56 & 58 && 87 & 34 & 86 & 88 \\
 28 & 27 & 29 & 81 && 58 & 21 & 57 & 59 && 88 & 51 & 87 & 89 \\
 29 & 28 & 30 & 38 &&59 & 58 & 60 & 68 && 89 & 8 & 88 & 90 \\
 30 & 13 & 29 & 31 && 60 & 43 & 59 & 61 && 90 & 1 & 73 & 89
\label{tab:F90 edges}
\end{longtable}
\begin{proof}
We note that $F090A$ is a cubic graph. Furthermore, $F090A$ is distance-regular, i.e. given any two ordered pairs $(u_{1},u_{2})$ and $(v_{1},v_{2})$ with $d(u_{1},u_{2})=d(v_{1},v_{2})$, there exists an automorphism of $F090A$ taking $u_{i}$ to $v_{i}$. A computer search yields the $*$-separated cutsets
\begin{align*}
C_{1}&=\{v_{1}, v_{4}, v_{7}, v_{10}, v_{15}, v_{21}, v_{26}, v_{29}, v_{32}, v_{35}, 
v_{41}, v_{47}, v_{54}, v_{61}, v_{65}, v_{68}, v_{71}, \\&
\vspace*{30 pt}v_{76}, v_{82}, v_{85}, v_{88}\},\\
C_{2}&=\{v_{1}, v_{4}, v_{7}, v_{10}, v_{15}, v_{21}, v_{26}, v_{29}, v_{33}, v_{36}, v_{41}, v_{45}, v_{48}, v_{55}, 
v_{61}, v_{65}, v_{68},\\&
\vspace*{30 pt} v_{76}, v_{80}, v_{85}, v_{88}\},\\
C_{3}&=\{v_{1}, v_{4}, v_{7}, v_{10}, v_{15}, v_{21}, v_{26}, v_{30}, v_{33}, v_{37}, v_{41}, v_{44}, v_{47}, v_{61}, v_{65}, v_{68}, v_{72},\\&
\vspace*{30 pt}v_{76}, v_{81}, v_{85}, v_{88}\}.
\end{align*}
Now, the vertices adjacent to $v_{1}$ are $v_{2},\;v_{18},\;v_{90}$: $v_{2}$ lies in a different component to $\{v_{18},v_{90}\}$ in $F090A-C_{2};$ $v_{18}$ lies in a different component to $\{v_{2},v_{90}\}$ in $F090A-C_{3};$ and $v_{90}$ lies in a different component to $\{v_{2},v_{18}\}$ in $F090A-C_{1}.$

Applying automorphisms we see that there exists a collection of $*$-separated vertex cutsets $C_{i}\subseteq V(F090A)$ such that $\vert C_{i}\vert \geq 2$ for each $i$, $\cup_{i}C_{i}=V(F090)$, and that for any vertex $v$ and any distinct vertices $w,w'$ adjacent to $v$ there exists a $*$-separated vertex cutset $C_{i}$ such that $w$ and $w'$ lie in distinct components of $F090A - C_{i}$.

If $x$ and $y$ are distance at least $3$ apart, we may apply an automorphism $\phi$ to map the pair $(x,y)$ to: $P_{3}=(v_{2},v_{17})$ if $d(x,y)=3$; $P_{4}=(v_{3},v_{19})$ if $d(x,y)=4$; $P_{5}=(v_{2},v_{9})$ if $d(x,y)=5$; $P_{6}=(v_{3},v_{9})$ if $d(x,y)=6$; $P_{7}=(v_{16},v_{39})$ if $d(x,y)=7$; and $P_{8}=(v_{16},v_{63})$ if $d(x,y)=8$. Each $P_{j}$ is separated by some $C_{i}$, and so $\phi^{-1}C_{i}$ separates the pair $x$ and $y$. Hence, $F090A$ is vertex $3$-separated.

Fix some cutset $C=C_{i}$, choose a vertex $v\in C$, and let $w\in V(F090A)$ be any vertex. Let $H=Aut(F090A)$. Since $H$ acts vertex transitively, there exists $h\in H$ such that $h v=w$. Therefore 
$$\{\gamma\in H \;:\;v\in \gamma C\}=\{\gamma\in H \;:\;w\in h\gamma C\}=\{h^{-1}\gamma'\in H \;:\;w\in \gamma' C\},$$
and hence $\vert\{\gamma\in H \;:\;v\in \gamma C\}\vert=\vert\{\gamma\in H \;:\;w\in \gamma C\}\vert.$ Let $$\tilde{\mathcal{C'}}:=\bigsqcup\limits_{C\in\{C_{1},C_{2},C_{3}\}}H(C),$$
with multiplicity. By the above, it follows that for any two vertices $v,w\in V(F090A)$,
$$\vert\{C\in \tilde{\mathcal{C}'} \;:\;v\in  C\}\vert=\vert\{C\in \tilde{\mathcal{C}'} \;:\;w\in  C\}\vert.$$ 

Finally, since $F090A$ is distance regular, for any $v$, $i\neq j$ and $i'\neq j'$ there exists some $\phi\in H$ with $\phi( *(v,i,j))=*(v,i',j').$
Therefore for any $C=C_{i}$, $\vert\{\gamma\in H \;:\;C\in *(v,i,j)\}\vert=\vert\{\gamma\in H \;:\;C\in  *(v,i',j')\}\vert.$ Let $\mathcal{C}$ be the underlying set of $\tilde{\mathcal{C}'}$. It follows immediately that there exists an integer $M$ such that for any vertex $v$ and any $i\neq j $, $\vert C\in\mathcal{C}(v,i,j)\vert  =M \slash 3.$
\end{proof}

It remains to show the following.

\begin{lemma}\label{lem: F090 complexes are separated}
Let $X$ be a simply-connected triangular complex such that every triangle is a unit equilateral Euclidean triangle, and the link of every vertex is isomorphic to $F090A$ as non-metric graphs. Then $X$ is CAT(0) and evenly $\pi$-separated.
\end{lemma}
\begin{proof}
As $girth(F090A)=10$, it follows by Gromov's link condition that $X$ is CAT(0) under this metric. The length of each edge in the link of a vertex is $\pi\slash 3$; since $F090A$ is vertex $3$-separated under the combinatorial metric, it follows that the link of each vertex is vertex $\pi$-separated under the CAT(0) metric.

Now, let $e$ be an oriented edge of $X^{(1)}$ with endpoints $v$ and $w$. 
Endow the links with the canonical partitions for each cutset. It is easily seen that every $C\in C(e,i,j)$ in $Lk(v)$ is equitable along $e$ to any $C'\in C(e^{-1},i',j')$ in $Lk(w)$ for an appropriate choice of $i',j'$. Furthermore, we note that, as $F090A$ is $*$-separated, the gluing equations are immediately solved by setting all weights equal to $1$, since:
	\begin{align*}
	    \sum\limits_{(C',P')\in [[C,P]]_{e}} \mu(C',P')&=\vert C(e,i,j)\vert\\& =M\slash 3\\&=\vert C(e^{-1},i',j')\vert \\&=\sum\limits_{(C',P')\in [C,P]_{e^{-1}}} \mu(C',P').
	\end{align*}
\end{proof}
Corollary \ref{maincor: application to F090A complexes} now follows immediately from Theorem \ref{mainthm: link conditions for the Haagerup property} and the above lemma.
	\bibliographystyle{alpha}
	\bibliography{references.bib}
{\sc{DPMMS, Centre for Mathematical Sciences, Wilberforce Road, Cambridge, CB3 0WB, UK}.} \textit{\textsc{E-mail address}}\textsc{: cja59@dpmms.cam.ac.uk}
\end{document}